\documentclass{article}
\usepackage{amsmath,amssymb,wasysym,amsfonts,amscd}
\usepackage{graphicx}
\usepackage{subfig}

\newtheorem{lemma}{Lemma}
\newtheorem{proposition}{Proposition}
\newtheorem{definition}{Definition}
\newtheorem{question}{Question}
\newenvironment{proof}{\noindent{\bf Proof.\,}}{\hfill$\Box$~\\}
\newtheorem{theorem}{Theorem}

\title{List precoloring extension in planar graphs}
\author{Maria Axenovich\thanks{Supported in part by NSA grant H98230-09-1-0063 and NSF grant DMS-0901008} \thanks{Department of Mathematics, Iowa State University, Ames, IA 50011} \and Joan P. Hutchinson\thanks{Department of Mathematics, Statistics and Computer Science, Macalester College, St Paul, MN 55105} \and Michelle A. Lastrina\footnotemark[2]}
\date{March 21, 2011}
\begin{document}
\maketitle

\begin{abstract}
A celebrated result of Thomassen states that not only can every planar graph be colored properly with five colors, but no matter how arbitrary  palettes of five colors are assigned to vertices, one can choose a color from the corresponding palette for each vertex so that the resulting coloring is proper. This result is referred to as $5$-choosability of planar graphs. Albertson asked whether Thomassen's theorem can be extended by precoloring some vertices which are at a large enough distance apart in a graph. Here, among others, we answer the question in the case when the graph does not contain short cycles separating precolored vertices and when there is a ``wide" Steiner tree containing all the precolored vertices.
\end{abstract}

\section{Introduction}
Let $G$ be a graph, let $L: V(G) \rightarrow 2^{\mathbb{N}}$ be an assignment of lists of colors to vertices of $G$. We say that a coloring $c:V(G) \rightarrow \mathbb{N}$ is an $L$-coloring, or {\it proper coloring from lists $L$} if $c(v)\in L(v)$ for $v\in V(G)$, and $c(u) \neq c(v)$ if $uv \in E(G)$. When such an $L$-coloring occurs, we say that $G$ is \textit{$L$-colorable}.  For extensive literature on list-colorings of planar graphs we refer the reader to \cite{ERT, M, T, T2, Vi, V}. Thomassen \cite{T} proved that if $G$ is a planar graph and $|L(v)|=5$ for each vertex $v\in V(G)$, then $G$ is $L$-colorable.  Let $P$ be a subset of vertices in a graph $G$.  We say that a {\it precoloring of P} is \textit{extendable to a $5$-list coloring} of $G$ if for every list assignment $L: V(G) \rightarrow 2^{\mathbb{N}}$, such that $|L(v)|=1$ for $v\in P$ and $|L(v)|=5$ for $v \in V(G)-P$, $G$ is $L$-colorable.  Albertson \cite{A} posed the following question.

Let $G$ be a plane graph.  Is there a $d>0$ such that whenever $P\subset V$ is such that the distance between every pair of vertices of $P$ is at least $d$, then every precoloring of $P$ extends to a $5$-list coloring of $G$?

Here, the distance \textit{dist$(x,y)$} between a pair of vertices $x$ and $y$ is the number of edges in a shortest path joining them.  Tuza and Voigt showed in \cite{TV}, see also \cite{VW}, that the condition of a large distance between precolored vertices is essential by finding a planar graph $G$ with a set of precolored vertices at pairwise distance at least $4$ such that the precoloring is not extendable to a $5$-list coloring of $G$.  So, the distance $d$ in the above question should be at least $5$.
Does this question have a positive answer if $d\geq 1000$?  The original theorem of Thomassen \cite{T} implies that if there are two adjacent precolored vertices assigned distinct colors, then the precoloring is extendable to a $5$-list coloring of $G$.  B\"{o}hme et al. \cite{BMS} described when the precoloring of vertices on a short face with at most six vertices can be extended to a $5$-list coloring of a planar graph.

In this manuscript, we introduce a technique using shortest paths in planar graphs which allows us to answer Albertson's question for a wide class of planar graphs. We prove that a proper precoloring of a pair of vertices can always be extended to a $5$-list coloring of a planar graph provided they are not separated by $3$- or $4$-cycles.  We also provide results about extensions of precolorings of vertices on one face.  Finally, we answer Albertson's question in the case where there are no $3$- or $4$-cycles separating precolored vertices and there is a special tree containing all of the precolored vertices.

To state our main results in all their generality, we need to define several notions.
We say a set of vertices $X$ separates a set of vertices $P$ in a connected graph $G$ if there are at least two vertices of $P$ contained in distinct connected components of $G-X$.  If $X$ is an $i$-vertex set separating $P$ in $G$ and spanning $C_i$, we say $G$ contains a $P$-separating $C_i$.
If $X$ separates $V(G)$ we say $X$ separates $G$, or $X$ is a separating set in $G$.
For a set of vertices $P$ in a graph $G$, \textit{dist$(P)=$ dist$(P,G)$} is the smallest distance in $G$ between two vertices of $P$.
For a path $S$, with endpoints $u$ and $v$, we say a vertex $w$ is central if the distances in $S$ from $w$ to $u$ and from $w$ to $v$ differ by at most $1$. Note there are at most two central vertices in $S$.  For graph theoretic terminology not defined here, we refer the reader to \cite{W}.

\begin{definition}
Let $G$ be a planar graph, $P$ a subset of vertices of $G$.  Fix a positive integer $d$.  Let $T$ be a tree with $P\subseteq V(T)$.  Let the set of special vertices be the union of $P$ and the set of vertices of degree either $1$ or at least $3$ in $T$.  A path in $T$ with special vertices as endpoints and containing no other special vertices is called a branch of $T$.

\noindent We say a tree $T$ is \textbf{$(P,d)$-Steiner} if \\
(1) every branch has length at least $2d$,\\
(2) every branch is a shortest (in $G$) path between its endpoints,\\
(3) if $v_c$ is a center of a branch of $T$, then a shortest (in $G$) path between $v_c$ and every vertex in another branch has length at least $d$, and \\
(4) no two vertices of $T$ from distinct branches have a common neighbor outside of $T$ nor are they adjacent.
\end{definition}

For example, when $P= \{u, v\}$ is a set of two vertices at distance $30$ from each other, a shortest $(u,v)$-path is a $\left(P, 15\right)$-Steiner tree with a single branch.

We say that a set $X$ of four vertices of degree at most $5$ in a graph $G$ forms the configuration $D=D(X)$, if $G[X]$ is isomorphic to $K_4-e$.  See Figure \ref{fig:D-reducible} for an illustration of $D$.
\begin{figure}[h]
\begin{center}
\subfloat[][Configuration $D$.]{\label{fig:D-reducible}\includegraphics{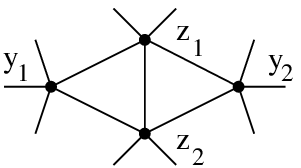}}
\qquad
\subfloat[][Configuration $W$.]{\label{fig:W-reducible}\includegraphics{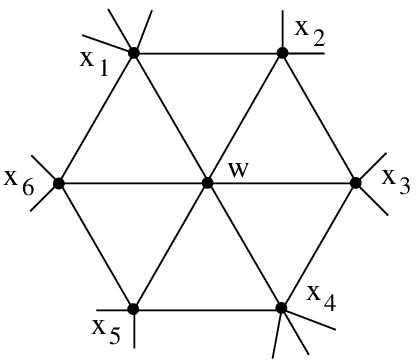}}
\end{center}
\caption{Reducible configurations.}
\label{fig:reducible}
\end{figure}

We say that a set $X$ of seven vertices of degree at most $6$ in a graph $G$ forms the configuration $W=W(X)$, if $G[X]$ induces a $6$-wheel, formed from a central vertex $w$ adjacent to a $6$-cycle $x_1, x_2, x_3, x_4, x_5, x_6, x_1$ such that $x_2,\,x_3,\,x_5$ and $x_6$ have degree at most $5$ in $G$.  See Figure \ref{fig:W-reducible} for an illustration of $W$.

For a graph $G$ and a set of precolored vertices $P$, let $R(G)=R(G, P)$, a \textbf{reduction of }$G$, be a graph obtained by one of the following operations:
(1) for a separating $C_3$ or $C_4$ that does not separate $P$, remove from $G$ the vertices and edges in the region that is bounded by the separating $C_3$ or $C_4$ and that does not contain any vertices of $P$,
(2) for a configuration $D=D(X)$ such that $P\cap X=\emptyset$, remove $X$ from $G$, or
(3) for a configuration $W=W(X)$ such that $P\cap X=\emptyset$, remove $X$ from $G$.
If none of these operations can be carried out, let $R(G)=G$.

Consider a sequence of graphs $G=G_0\supset G_1 \supset G_2 \supset \cdots \supset G_m$ such that $G_i=R(G_{i-1}, P)$ for $i=1, \ldots, m$ and $R(G_m)=G_m$.
Call such a graph $G_m$ a {\bf reduced} graph of $G$.
A reduced graph does not have a separating $C_3$ or $C_4$ that does not separate $P$ and it contains no configurations $D(X)$ or $W(X)$ with $P\cap X=\emptyset$.
We shall show that if a reduced graph of $G$ has a coloring extension of $P$, then so does $G$.

We now state the main results of this manuscript.

\begin{theorem} \label{Steiner}
Let $G$ be a plane graph, let $P$ be a set of  vertices such that there is no $P$-separating $C_3$ or $C_4$ in $G$.
If there is a reduced graph of $G$ that has a  $(P,45)$-Steiner tree, then every precoloring of $P$ is extendable to a proper $5$-list coloring of $G$.
\end{theorem}

\begin{theorem}\label{two}
Let $G$ be a plane graph and $u, v\in V(G)$.  If $G$ has no $\{u, v\}$-separating $C_3$ or $C_4$, then every proper precoloring of $\{u,v\}$ is extendable to a proper $5$-list coloring of $G$.
\end{theorem}

\begin{theorem}\label{one-face}
Let $G$ be a plane graph and $C$ a set of vertices of a facial cycle of $G$. Let $P = \{v_0, v_1, \ldots, v_{k-1}\}\subseteq  C$, where the vertices of $P$ are labeled cyclically around $C$.
Then every proper precoloring of $P$ is extendable to a $5$-list coloring of $G$ if one of the following conditions holds:
\begin{enumerate}
\item $G[P]$ consists of disjoint vertices and edges with pairwise distance at least $3$,
\item $k\leq 6$ and none of the following occur:
    \begin{enumerate}
    \item There is a vertex $u \in V(G) - P$ adjacent to at least five vertices of $P$ such that $L(u)$ consists of the colors assigned to those five vertices.
    \item $k = 6$ and there is an edge $u_0u_1$  and a color $\alpha$ such that, for $i=0,1$, the vertex $u_i$ is adjacent to $v_{3i+1},v_{3i+2},v_{3i+3},v_{3i+4}$, where addition of indices is modulo $k$, and $L(u_i)$ consists precisely of the colors assigned to those four vertices and $\alpha$.
    \item $k=6$ and there is a triangle $(u_0,u_1,u_2)$ and colors $\alpha,\beta$, such that, for $i=0,1,2$, the vertex $u_i$ is adjacent to $v_{2i+1},v_{2i+2},v_{2i+3}$, where addition of indices is modulo $k$, and $L(u_i)$ consists of the colors assigned to those three vertices and $\alpha,\beta$.
    \end{enumerate}
\end{enumerate}
\end{theorem}

\begin{theorem}\label{two-faces}
Let $P$ be a set of vertices in a plane graph $G$, $dist(P)\ge3$, such that there are two faces $F_1,F_2$ where the vertices of $P$ lie on the boundaries of $F_1$ and $F_2$.
Assume $G$ contains no $P$-separating $C_3$ or separating $C_4$.
Then every precoloring of $P$ is extendable to a proper $5$-list coloring of $G$.
\end{theorem}

The rest of the paper is organized as follows. In Section \ref{Preliminaries} we state known results mentioned above in detail and prove some technical lemmas. We prove all of the theorems in Section \ref{Proofs}. Finally, we state open problems and comments in Section \ref{Conclusions}.

\section{Preliminaries} \label{Preliminaries}
\begin{theorem}[Thomassen's $5$-list coloring theorem \cite{T},\cite{T3}]\label{T}
Let $G$ be a plane graph, $F$ the set of vertices of a face of $G$,  $L: V(G)\rightarrow 2^{\mathbb N}$ an assignment of lists of colors to vertices of $G$ such that $|L(w)| = 5$ for all $w\not\in F$, $|L(u)|=|L(v)|=1$ with $L(u)\neq L(v)$ for some adjacent vertices $u, v\in F$, and $|L(w)|=3$ for $w\in F- \{u,v\}$.  Then $G$ is $L$-colorable.
\end{theorem}

\begin{theorem}[Tuza-Voigt \cite{TV}]
There is a planar graph $G$ and a set $P$ of vertices with $dist(P) \geq 4$ and an assignment of lists of size $3$ to vertices of $P$ and lists of size $5$ to the remaining vertices such that $G$ is not colorable from these lists.
\end{theorem}

\begin{theorem}[B\"{o}hme, et al. \cite{BMS}]\label{bms}
Let $G=(V,E)$ be a plane graph with facial cycle $C$ of length $k\leq6$, where the vertices of $C$ are labeled cyclically $v_0, v_1, \ldots, v_{k-1}$.  If $|L(v)|=1$ for $v\in V(C)$, $|L(v)|=5$ for $v\in V-V(C)$, and $G[V(C)]$ is $L$-colorable, then $G$ is $L$-colorable unless one of the following occurs:
\begin{enumerate}
\item There is a vertex $u\in V-V(C)$ adjacent to five vertices in $C$ and $L(u)$ consists exactly of the colors assigned to those five vertices.
\item $k=6$ and there is an edge $u_0u_1$, $u_0, u_1\not \in V(C)$  and a color $\alpha$ such that, for $i=0,1$, the vertex $u_i$ is adjacent to $v_{3i+1},v_{3i+2},v_{3i+3},v_{3i+4}$, where addition of indices is modulo $k$, and $L(u_i)$ consists precisely of the colors assigned to those four vertices and $\alpha$.
\item $k=6$ and there is a triangle $(u_0,u_1,u_2)$, $u_0, u_1, u_2\not\in V(C)$ and colors $\alpha,\beta$, such that, for $i=0,1,2$, the vertex $u_i$ is adjacent to $v_{2i+1},v_{2i+2},v_{2i+3}$, where addition of indices is modulo $k$, and $L(u_i)$ consists of the colors assigned to those three vertices and $\alpha,\beta$.
\end{enumerate}
\end{theorem}

See Figure \ref{fig:onefaceforb} for illustrations of the forbidden configurations described in conditions 2 and 3 of Theorem \ref{bms}.  Note how these compare to the forbidden configurations of Theorem \ref{one-face}.
\begin{figure}[h]
\begin{center}
\subfloat[][Forbidden configuration of Theorem \ref{bms} condition 2.]{\label{fig:forbedge}\includegraphics{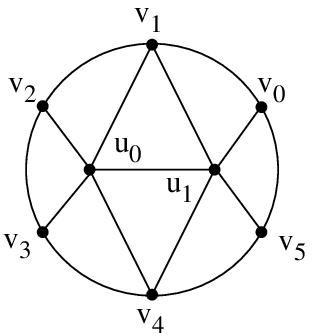}}
\qquad
\subfloat[][Forbidden configuration of Theorem \ref{bms} condition 3.]{\label{fig:forbtri}\includegraphics{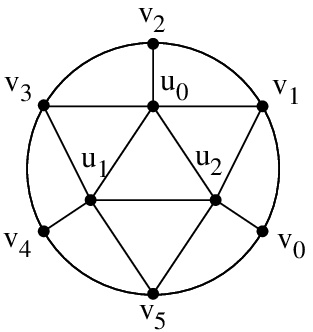}}
\end{center}
\caption{Forbidden configurations.}
\label{fig:onefaceforb}
\end{figure}

For a vertex $v\in V(G)$, we denote $N(v)=N(v,G)$ the neighborhood of $v$ in $G$.
For a vertex $v$ and a set of vertices $X$, we write $v\sim X$  if $v$ is adjacent to all vertices in $X$.
Let $H$ be a subgraph of $G$ and $c$ be a vertex coloring of $H$.
For $v \not\in V(H)$, let $c(v,H) = \{ c(u): u \in N(v) \cap V(H)\}$ be the set of colors used on neighbors of $v$ in $H$.
Let $d(v,H) = |N(v) \cap V(H)|$ be the size of the neighborhood of $v$ in $H$.
For a vertex set $X$ in $G$,  let $N(X) = N(G[X])$ be the set of neighbors of vertices from $X$ not in $X$.
For an induced subgraph $H$ of $G$ and $v\in V(G)-V(H)$, let $L_c(v, H) = L(v) - c(v,H)$.  When the subgraph $H$ is clear, we use $L_c(v)$.
We say $H$ is colored {\it {\bf  nicely}} by a coloring $c$ with respect to lists $L$ if $c$ is an $L$-coloring of $H$ and for every vertex $v\in N(H)$, $|L_c(v,H)|\geq 3.$
We also say $c$ is a {\it nice} coloring of $H$ in this case.
A vertex from $N(H)$ adjacent to at least three vertices in $H$ is called a {\it three-neighbor}, or simply $3$-neighbor, of $H$.
We denote the set of $3$-neighbors of $H$ by $N_3(H)$.
\begin{definition}
Let $Q(H)=G[H\cup N_3(H)]$ be the subgraph of $G$ induced by vertices of $H$ and its $3$-neighbors.
\end{definition}
For a path $S= v_0, v_1, \ldots, v_m$, and two vertices $v_i, v_j$ of $S$ we write $v_iSv_j$ to denote the subpath $v_i, v_{i+1}, \ldots, v_{j-1}, v_j$ of $S$.

The following proposition is almost identical to Thomassen's theorem 5.3 of \cite{T3}, with the added condition that $H$ contains all precolored vertices.  The proof is included for completeness.

\begin{proposition}\label{nicelyH}
Let $G$ be a planar graph and $P$ a set of vertices.  Let $L$ be an assignment of lists of colors such that $|L(v)|=1$ for $v\in P$ and $|L(v)|=5$ for $v\in V(G)-P$.
If there is an induced connected subgraph $H$ of $G$ containing all vertices from $P$ such that it can be nicely colored with respect to $L$, then $G$ is $L$-colorable.
\end{proposition}

Note if $d(v,H)\leq 2$ for each $v\not\in V(H)$ then every proper coloring of $H$ is a nice coloring.

\begin{proof}
Consider a nice coloring $c$ of $H$.  Then $|L_c(v, H)|\geq3$ for all $v\in N(H)$ and $|L_c(v,H)|=5$ for all $v\in G-V(H)$.
Therefore, by Thomassen's theorem, $G-V(H)$ is $L_c$-colorable. Together with the coloring $c$ of $H$, this gives a proper $L$-coloring of $G$ as $L_c(v)\subset L(v)$ for all $v\in G-V(H)$.
\end{proof}

\begin{lemma} \label{properties-shortest-path}
Let $S$ be a shortest $(u,v)$-path in a planar graph $G$, where $S= v_0, v_1, \ldots, v_m$ with $u=v_0$, $v=v_m$. Then the following properties hold:\\
(1) for all $w\in N(S)$, $d(w,S) \leq 3$,\\
(2) for every $x,y\in V(S)$, $x\not\sim y$ in $G$ unless $\{x,y\} =\{v_i, v_{i+1}\}$, for $i=0, \ldots, m-1$,\\
(3) if $d(w, S)=3$ for some $w\in N(S)$, then $w\sim \{ v_i,v_{i+1},v_{i+2}\}$, for $i=0, 1, \ldots, m-2$,\\
(4) if there is no separating $C_3$ or $C_4$ in $G$, then for each $i$ with $i=0, 1, \ldots, m-2$ there is at most one vertex $w\in N(S)$ such that  $w\sim \{ v_i,v_{i+1},v_{i+2}\}$.
\end{lemma}

\begin{proof}
Items (1)-(3) hold because $S$ is a shortest $(u,v)$-path.
To see the validity of item (4), assume there are two vertices adjacent to $v_i, v_{i+1}, v_{i+2}$. Then it is easy to verify that there is either a separating $C_3$ or  a separating $C_4$ in $G$.
\end{proof}

Note that Lemma \ref{properties-shortest-path} implies that if $S$ is a shortest path between two vertices of a planar graph $G$, then every block of $Q(S)$ with at least three vertices of $S$ has vertex set $\{v_i, v_{i+1}, \ldots, v_{i+k}, w_{i+1}, w_{i+2}, \ldots, w_{i+k-1}\}$, for some $k\ge2$, where $v_i, \ldots, v_{i+k}$ are consecutive vertices of $S$ and $w_{i+j} \sim \{v_{i+j-1}, v_{i+j}, v_{i+j+1}\}$, for $j = 1, \ldots, k-1$.
Observe that because $S$ is a shortest path and there are no separating $C_3$s or $C_4$s in $G$, the vertices of $Q(S)-S$ form an independent set.
We call a block of $Q(S)$ with $i$ vertices of  $S$ an \textbf{i-block}, $i=2, 3, 4, \ldots$.  See Figure \ref{fig:blocks} for examples of blocks in $Q(S)$. Note also that the block-cut-vertex tree of $Q(S)$ is a path.
\begin{figure}[h]
\begin{center}
\includegraphics{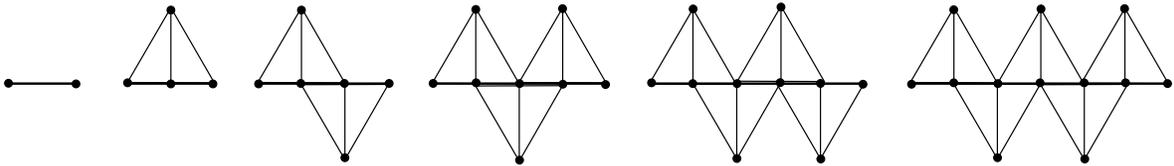}
\caption{Blocks in $Q(S)$, where the bold line indicates $S$.}
\label{fig:blocks}
\end{center}
\end{figure}
Note that if $Q(S)$ has a cut-edge, that edge is in $S$, and if $Q(S)$ has a cut-vertex, that vertex is in $S$.
We shall need a notion of a nontrivial block which will allow us to focus on subpaths of $S$ and not worry about the boundary conditions.
For a shortest $(u',v')$-path $T'$, we say an edge $e$ is a \textbf{nontrivial cut-edge} of $Q(T')$ if $e$ is a cut-edge not incident to either $u'$ or $v'$; we say $B$ is a \textbf{nontrivial block} of $Q(T')$ if $B$ is a block that does not contain $u'$ or $v'$.  We say a block $B$ is a \textbf{remote nontrivial block} of $Q(T')$ if $|V(B)\cap V(B_1)|=|V(B)\cap V(B_2)|=1$ where $B_1$ and $B_2$ are distinct nontrivial blocks of $Q(T')$.
Let $u', v' \in V(S)$ and let $T' = u'Sv'$. If $e$ is a nontrivial cut-edge in $Q(T')$, then it is easy to see that $e$ is a cut-edge in $Q(S)$; if $B$ is a nontrivial block of $Q(T')$, then $B$ is a block of $Q(S)$.

\begin{lemma}\label{nice1}
Let $S$ be a shortest $(u,v)$-path in a planar graph $G$, where $S= v_0, v_1, \ldots, v_m$ with $u=v_0$ and $v=v_m$.
Let $G$ have no separating $C_3$ and no separating $C_4$.
Let $L: V(G) \rightarrow 2^{\mathbb N}$ be an assignment of lists of colors to the vertices of $G$ where $|L(u)|=|L(v_1)|=1$ with $L(u)\neq L(v_1)$ and $|L(x)|=5$ for $x \in S\cup N(S) -\{u, v_1\}$.  Then $S$ can be nicely colored with respect to $L$.
\end{lemma}

\begin{proof}
Assume $v_0Sv_{i+1}$, where $i+2 \leq m$, has been colored nicely by $c$ and $c(v_i)=1$, $c(v_{i+1})=2$.
If there is no $w\in N(S)$ such that $w\sim\{v_i,v_{i+1},v_{i+2}\}$, then color $v_{i+2}$ arbitrarily from its list so that $c(v_{i+1})\neq c(v_{i+2})$.
If there is a $w\in N(S)$ such that $w\sim\{v_i,v_{i+1},v_{i+2}\}$, choose a color for $v_{i+2}$ more carefully.
If $1$ or $2$ is not in $L(w)$, then choose $c(v_{i+2})$ from $L(v_{i+2})-\{2\}$.
Otherwise, $L(w)=\{1,2,\alpha,\beta,\gamma\}$, for some colors $\alpha, \beta, \gamma$.
If $1\in L(v_{i+2})$, let $c(v_{i+2})=1$.
If $1\not\in L(v_{i+2})$, then there is $a\in L(v_{i+2})-L(w)$. Let $c(v_{i+2})=a$.
In each case, we have constructed a nice coloring of $v_0Sv_{i+2}$.
Since $|L(v_j)|=5$ for $j=2,\ldots,m$, the above argument may be applied along $S$ up through $v$ so that $S$ is nicely colored.
\end{proof}

\begin{lemma}\label{nice2}
Let $S$ be a shortest $(u,v)$-path in a planar graph $G$, where $S=v_0, v_1, v_2, \ldots, v_m$ with $u=v_0$ and $v=v_m$.
Let $G$ have no separating $C_3$ and no separating $C_4$. Let $L: V(G) \rightarrow 2^{\mathbb N}$ be an assignment of lists of colors to the vertices of $G$ where $|L(u)|=|L(v)|=4$ and $|L(x)|=5$ for $x \in S\cup N(S) -\{u, v\}$.
Then $S$ can be nicely colored with respect to $L$.
\end{lemma}

\begin{proof}
The proof is by induction on $|V(S)|$.\\
If $|V(S)| \leq 2$, the statement follows trivially.
If $|V(S)|=3$, we can assume that there is a vertex $w$, where $w\sim \{v_0, v_1, v_2\}$, otherwise color $S$ properly from $L$.  If $c_0\in L(v_0) \cap L(v_2)$ for some $c_0$, let $c(v_0)=c(v_2)=c_0$ and color $v_1$ arbitrarily from $L(v_1)-\{c_0\}$.  If $L(v_0) \cap L(v_2) = \emptyset$, then $|L(v_0)\cup L(v_2)|=8$ and there is a color $c_0 \in (L(v_0)\cup L(v_2)) -L(w)$. Assume without loss of generality that $c_0 \in L(v_0)$. Then let $c(v_0)=c_0$, and color $v_1, v_2$ arbitrarily from their lists so the path $v_0, v_1, v_2$ is properly colored. As a result $|L_c(w)| \geq 3$.

Now assume the result holds for shortest paths on fewer than $m+1$ vertices.  Let $|V(S)|=m+1$.
Color $v_0Sv_{m-1}$ nicely with a coloring $c$.
If there is no vertex outside of $S$ adjacent to $v_{m-2}, v_{m-1}$ and $v$, then choose $c(v)$ from $L(v)-\{c(v_{m-1})\}$.
This gives a nice coloring of $S$.
So assume there is a vertex $w\in N(S)$ such that $w\sim\{v_{m-2},v_{m-1},v\}$.

If $c(v_{m-1})\not\in L(w)$ or $c(v_{m-2})\not\in L(w)$, then let $c(v)\in L(v)-\{c(v_{m-1})\}$.  If $c(v_{m-1})\in L(w)-L(v)$ and $c(v_{m-2})\in L(w)-L(v)$, then $L(v)$ contains a color not in $L(w)$.  Assign this color to $v$ to obtain a nice coloring of $S$.  So we can assume $c(v_{m-2})$ or $c(v_{m-1})\in L(v)\cap L(w)$.
If $c(v_{m-2}) \in L(v)$, let $c(v)= c(v_{m-2})$ providing a nice coloring of $S$.
Thus we can assume $L(v) = \{c_1, c_2, c_3, c_4\}$, $a=c(v_{m-2}) \neq c_i$ for all $i=1,2, 3, 4$, and $L(w) = \{a, c_1, c_2, c_3, c_4\}$.

Apply induction to $v_0Sv_{m-2}$ in the graph $G'$ induced in $G$ by this path and its neighbors, with a new list $L(v_{m-2}) -\{a\}$  assigned to $v_{m-2}$ and all other old lists.
There is a nice coloring $c'$ of $v_0Sv_{m-2}$ in $G'$. Note that it is a nice coloring of $v_0Sv_{m-2}$ in $G$.
We either have $c'(v_{m-2}) = c_i$, for some $i=1,2, 3, 4$, or $c'(v_{m-2} )\not \in L(w)$. If $c'(v_{m-2})\not\in L(w)$, color $v_{m-1}$ first so that if there is $w'\sim\{v_{m-3},v_{m-2},v_{m-1}\}$, then $|L_{c'}(w')|\ge3$.  Then let $c'(v)\in L(v)-\{c'(v_{m-1})\}$.  If $c'(v_{m-2})=c_i$, without loss of generality say $c'(v_{m-2})=c_1$, then let $c'(v)=c_1$ and color $v_{m-1}$ so that $|L_{c'}(w')|\ge3$.
\end{proof}

\begin{lemma}\label{nice3}
Let $S$ be a shortest $(u,v)$-path in a planar graph $G$, where $S= v_0, v_1, \ldots, v_m$ with $u=v_0$ and $v=v_m$.
Let $G$ have no separating $C_3$ and no separating $C_4$.
Let $L: V(G) \rightarrow 2^{\mathbb N}$ be an assignment of lists of colors to the vertices of $G$ with $|L(u)|=|L(v)|=1$ and $|L(x)|=5$ for $x \in S\cup N(S) -\{u, v\}$. Assume $Q(S)$ has at least two cut-edges.  Then $S$ can be nicely colored with respect to $L$.
\end{lemma}

\begin{proof}
Let $v_k v_{k+1}$ and $v_l v_{l+1}$ be two cut-edges of $Q(S)$, where $0\leq k<l<m$.
Using Lemma \ref{nice1} color $v_0Sv_k$ and $v_{l+1}Sv_m$ nicely with a coloring $c$.
If $v_{k+1} = v_l$, we are done by giving $v_l$ a color different from $c(v_k)$ and $c(v_{l+1})$.
Otherwise, delete $c(v_k)$ from $L(v_{k+1})$, delete $c(v_{l+1})$ from $L(v_l)$ and color $v_{k+1}Sv_l$ nicely from the updated lists using Lemma \ref{nice2}.  Since $v_0Sv_k$, $v_{k+1}Sv_l$, and $v_{l+1}Sv_m$ do not have pairwise common neighbors in $N_3(S)$, this gives a nicely colored $S$.
\end{proof}

The next lemma is a key lemma in this paper, stating that either a given shortest path between two precolored vertices could be nicely colored, or another subgraph that is close to that path could be nicely colored.

For a path $T'$, a center $v_c$ of $T'$, and an even positive integer $d\le |V(T')|-1$, we call the two vertices of $T'$ at distance (in $T'$) $\frac{1}{2}d$ from $v_c$ the \textbf{$d$-tag vertices} with respect to $v_c$, or simply tag vertices, of $T'$.

\begin{lemma} \label{nice4}
Let $S$ be a shortest $(u,v)$-path in a planar graph $G$ with a center $v_c$ and $40$-tag vertices $u^*$, $v^*$ with respect to $v_c$, where $S=v_0,v_1,\ldots,v_m$ with $v_0=u$ and $v_m=v$.
Assume $G$ contains no separating $C_3$ or $C_4$ and no configuration $D(X)$ or $W(X)$ with $\{u,v\}\cap X=\emptyset$.  Let $L: V(G) \rightarrow 2^{\mathbb N}$ be an assignment of lists of colors to the vertices of $G$ such that $|L(u)|=|L(v)|=1$ and $|L(x)|=5$ for $x\in V(G)-\{u,v\}$.
Then there is a connected graph $H = H(S, u, v) = uSu^*\cup H' \cup v^*Sv$ such that every vertex of $H'$ is at distance at most $21$ from $v_c$, and $H$ can be nicely colored from $L$.
\end{lemma}

\begin{proof}
Recall that every vertex in $V-V(S)$ is adjacent to at most three vertices in $S$, and if a vertex from $Q(S)-S$ is adjacent to vertices in $S$, these vertices in $S$ must be consecutive.

\noindent \textit{Observation 1.}  If $Q(S)$ has a $p$-block $B$, for a $p\geq 6$, then there is a shortest $(u,v)$-path $S'$ such that $Q(S')$ has a nontrivial cut-edge.

Let $B$ contain $v_i, v_{i+1}, \ldots, v_{i+5}$ and vertices $w_{i+k}$ not in $V(S)$, where $w_{i+k}\sim \{v_{i+k-1}, v_{i+k}, v_{i+k+1}\}$, for $k=1, 2, 3,4$.
Consider the shortest $(u,v)$-path
$$S' = v_0, v_1, \ldots, v_i, w_{i+1},  v_{i+2},  v_{i+3}, w_{i+4}, v_{i+5}, \ldots, v_m.$$
Then it is a routine check to see that $v_{i+2}v_{i+3}$ is a nontrivial cut-edge in $Q(S')$, as shown in Figure \ref{fig:observation1}.
\begin{figure}[h]
\begin{center}
\includegraphics{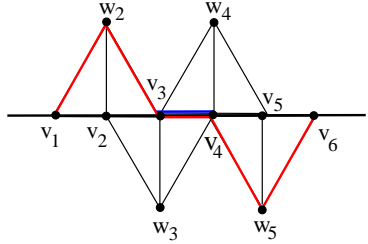}
\end{center}
\caption{Note that $v_{i+2}v_{i+3}$ is a nontrivial cut-edge in $Q(S')$.}
\label{fig:observation1}
\end{figure}
%v_{i+2}v_{i+3} is a cut-edge: Assume there is $w\not\in V(Q(S))$ such that $w\sim\{v_{i+2}.v_{i+3}\}$.  It cannot lie in any of the regions bounded by the triangles formed by vertices in $Q(S)$, otherwise there will be a separating $C_3$.  Also, $w$ cannot lie outside the boundary of $Q(S)$, otherwise there will be a $P$-separating $C_3$ or we contradict the planarity of $G$.

\noindent \textit{Observation 2.} We can assume at least one of the following holds:\\
(1) for every shortest $(u^*,v_c)$-path $S'$, each nontrivial block of $Q(S')$ is either a $3$-, $4$-, or $5$-block, \\
(2) for every shortest $(v_c,v^*)$-path $S'$, each nontrivial block of $Q(S')$ is either a $3$-, $4$-, or $5$-block.

If there is a shortest $(u^*,v_c)$-path $T'$ such that $Q(T')$ has a nontrivial cut-edge  and there is a shortest $(v_c, v^*)$-path $T''$ such that $Q(T'')$ has a nontrivial cut-edge, then Lemma \ref{nice3} implies $uSu^*T'v_cT''v^*Sv$ can be nicely colored.
Assume, without loss of generality that for every shortest $(u^*,v_c)$-path $S'$, $Q(S')$ has no nontrivial cut-edges.
Then Observation 1 implies there is no $p$-block of $Q(S')$ with $p\geq 6$.

\noindent Assume that part (1) of Observation 2 holds. Let $u'=u^*$, $v'= v_c$.
Let $T$ be a shortest  $(u', v')$-path with the largest number of $3$-neighbors.

\noindent \textit{Observation 3.}  If, for some shortest $(u',v')$-path $T$ with maximum number of $3$-neighbors,  $Q(T)$ has a nontrivial $3$-block, then there is a graph $H(S,u,v)$ satisfying the conditions of the lemma.

If such a block $B$ were to exist, say with consecutive vertices $x_i, x_{i+1}, x_{i+2}$ of $T$ and $w \not \in V(T)$, $w\sim \{x_i, x_{i+1}, x_{i+2}\}$, then there is no vertex $w' \not \in S\cup T$, such that $w'$ is adjacent to $w$ and two other vertices of $T$, otherwise there is a shortest $(u',v')$-path with more $3$-neighbors than $T$.
Let $H(S, u, v)$ be the graph induced by vertices of $uSu'Tv'Sv$ and $w$.
Nicely color $uSu'Tx_i$ and nicely color $x_{i+2}Tv'Sv$, then properly color $w$ and $x_{i+1}$ from remaining available colors in their lists.
Since there is no $3$-neighbor of $H(S, u, v)$ adjacent to $x_{i+1}$ and there is no such $3$-neighbor adjacent to $w$, this coloring is a nice coloring of $H(S,u,v)$.

\noindent Thus, we can assume that all nontrivial blocks of $Q(T)$ are $4$- or $5$-blocks. Since $dist(u',v') = 20$, there are remote nontrivial blocks in $Q(T)$.

\noindent \textit{Observation 4.} If $Q(T)$ has a remote $5$-block for some shortest $(u',v')$-path $T$ with maximum number of $3$-neighbors, then there is a graph $H(S,u,v)$ satisfying the conditions of the lemma.

\noindent Assume there is such a block $B$ with consecutive vertices $x_i, x_{i+1}, x_{i+2}, x_{i+3}, x_{i+4}$ of $T$ and vertices $w_{i+1}$, $w_{i+2}$,$w_{i+3}$ not in $T$ such that $w_k \sim \{x_{k-1},x_k, x_{k+1}\}$, for $k =i+1, i+2, i+3$.
From Observation 3, we can assume that every nontrivial block of $Q(T)$ is either a $4$- or a $5$-block.

Note first that there is no vertex $w$ adjacent to $w_{i+1}$ and two vertices of $T$, and there is no vertex $w$ adjacent to $w_{i+3}$ and two vertices of $T$.
Indeed, assume otherwise that  there is a vertex $w$ adjacent to $w_{i+1}$ and two vertices of $T$.
Then $w\sim \{x_{i-1}, x_i, w_{i+1}\}$.
Since all nontrivial blocks of $T$ have at least four vertices of $T$, and there are nontrivial blocks $B_1$ and $B_2$ of $Q(T)$ such that $|V(B)\cap V(B_1)|=|V(B)\cap V(B_2)|=1$, we see there is a vertex $w_{i-1}$ adjacent to $\{x_{i-2}, x_{i-1}, x_{i}\}$ and there is a vertex $w_{i-2}$ adjacent to $\{x_{i-3}, x_{i-2}, x_{i-1}\}$, as shown in Figure \ref{fig:obs4case2}.
\begin{figure}[h]
\begin{center}
\includegraphics{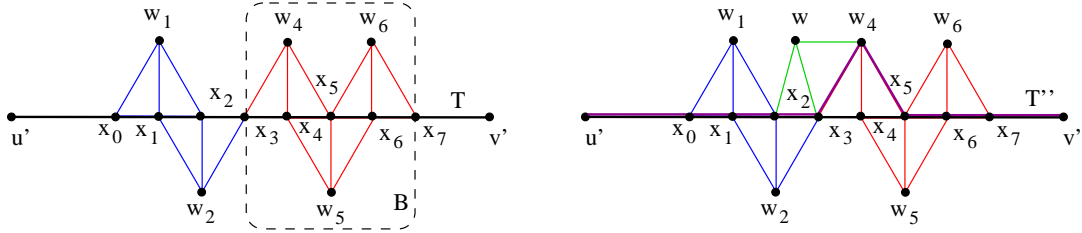}
\end{center}
\caption{Example corresponding to a case of Observation 4.}
\label{fig:obs4case2}
\end{figure}
Then $u'T x_{i-3}, x_{i-2}, x_{i-1}, x_i, w_{i+1}, x_{i+2} Tv'$ is a shortest $(u',v')$-path  $T''$ with a block in $Q(T'')$ having at least six vertices of $T''$, as can be seen in Figure \ref{fig:obs4case2}. This is a contradiction to Observation 2.
Similarly, it is impossible to have a vertex $w$ adjacent to $w_{i+3}$ and two vertices of $T$.

Assume now that there is no vertex $w$ adjacent to $w_{i+1}$ and  $w_{i+3}$ and a vertex of $T$.
Let $H(S, u, v)$ be a graph induced by vertices of $uSu'Tv'Sv$ and $w_{i+1}, w_{i+3}$, as shown in Figure \ref{fig:obs4case1}.  Note that while $w_{i+2}$ is shown in the figure, it is not a vertex in the graph $H(S,u,v)$.
\begin{figure}[h]
\begin{center}
\includegraphics{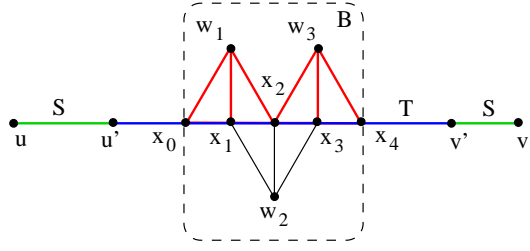}
\end{center}
\caption{An example of $H(S,u,v)$, as described in a case of Observation 4.}
\label{fig:obs4case1}
\end{figure}
To color $H(S, u,v)$ nicely, first color $uSu'Tx_{i}$ and $x_{i+4}Tv'Sv$ nicely, then color $x_{i+1}Tx_{i+3}$ properly so $w_{i+2}$ has at least three colors remaining in its list after the removal of colors used on adjacent vertices, and finally color $w_{i+1}$ and $w_{i+3}$ using available colors.

\noindent\textit{Fact.} We can assume for every shortest $(u',v')$-path $T=y_0,y_1,\ldots,y_l$, where $y_0=u'$, $y_l=v'$, with maximum number of $3$-neighbors and for every remote nontrivial $5$-block $B$ of $T$ with vertices $y_i,y_{i+1},y_{i+2},y_{i+3},y_{i+4}$ of $T$ and $w_j\sim\{y_{j-1},y_j,y_{j+1}\}$, for $j=i+1,i+2,i+3$, there is a vertex $w\sim\{w_{i+1},y_{i+2},w_{i+3}\}$.

Consider the shortest $(u',v')$-path $\tilde{T}_1=u'Ty_i,y_{i+1},w_{i+2},y_{i+3},y_{i+4}Tv'$.  There must be a vertex $w'\sim\{y_i,y_{i+1},w_{i+2}\}$, otherwise $y_iy_{i+1}$ is a cut-edge in $Q(\tilde{T}_1)$.  There must also be a vertex $w''\sim\{w_{i+2},y_{i+3},y_{i+4}\}$, otherwise $y_{i+3}y_{i+4}$ is a cut-edge in $Q(\tilde{T}_1)$.

Next, consider the shortest $(u',v')$-path $\tilde{T}_2=u'Ty_i,w_{i+1},w,w_{i+3},y_{i+4}Tv'$.  There must be a vertex $x\sim\{y_i,w_{i+1},w\}$, otherwise $y_iw_{i+1}$ is a cut-edge in $Q(\tilde{T}_2)$.  There must also be a vertex $x'\sim\{w,w_{i+3},y_{i+4}\}$, otherwise $w_{i+3}y_{i+4}$ is a cut-edge in $Q(\tilde{T}_2)$.
Finally, consider the shortest $(u',v')$-paths
$$\tilde{T}_3=u'Ty_i,y_{i+1},w_{i+2},y_{i+3},y_{i+4}Tv' \text{ and } \tilde{T}_4=u'Ty_i,w_{i+1},w,w_{i+3} ,y_{i+4}Tv'.$$  By the fact above, there must be vertices $z$ and $z'$ such that $z\sim\{w',w_{i+2},w''\}$ and $z'\sim\{x,w,x'\}$.
Thus, $G[X]$ where $X=\{y_{i+1},y_{i+2},y_{i+3},w_{i+1},w_{i+2},w_{i+3},w\}$ corresponds to the configuration $W(X)$ in $G$, as seen in Figure \ref{fig:obs4case3}, where the bold vertices represent $X$.
\begin{figure}[h]
\begin{center}
\includegraphics{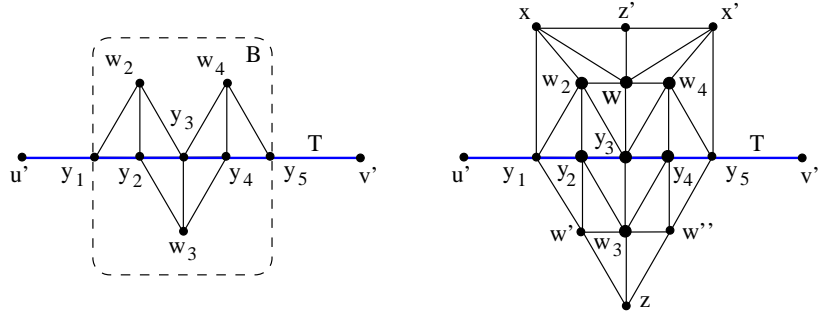}
\end{center}
\caption{The configuration $W$ as it arises locally around $T$.}
\label{fig:obs4case3}
\end{figure}
This completes the proof of Observation 4.

\noindent To summarize, we know that for any shortest $(u',v')$-path $T$, every nontrivial block of $Q(T)$ is a $3$-, $4$-, or $5$-block.
Moreover, if $T$ has the largest number of $3$-neighbors among all such shortest $(u', v')$-paths, then every remote nontrivial block of $Q(T)$ is a $4$-block.

\noindent To conclude the proof of Lemma \ref{nice4}, let $T$ be a shortest $(u',v')$-path with the largest number of $3$-neighbors among all such shortest $(u',v')$-paths.  Consider a remote nontrivial block of $Q(T)$ with consecutive vertices $x_i, x_{i+1}, x_{i+2}, x_{i+3}$ of $T$ and vertices $w_{i+1}, w_{i+2}$ not in $T$ such that $w_k \sim \{x_{k-1}, x_{k}, x_{k+1}\}$ for $k =i+1,i+2$.

\noindent Case 1. There is no $w$ adjacent to $w_{i+1}$ and two vertices of $T$, and there is no vertex $w$ adjacent to $w_{i+2}$ and two vertices of $T$.

Let $H(S, u, v)$ be the graph induced by vertices of  $uSu'Tv'Sv$ and $w_{i+1}, w_{i+2}$.
To color $H(S, u,v)$ nicely, first color $uSu'Tx_{i}$ and $x_{i+3}Tv'Sv$ nicely, then color $G[x_{i+1}, x_{i+2}, w_{i+1}, w_{i+2}]$ properly.

\noindent Case 2. There is, without loss of generality, a vertex $w$ adjacent to $w_{i+1}$ and two vertices of $T$.

If $w'\sim\{x_{i-1},x_i,w_{i+1}\}$ for some vertex $w'$, then consider the path $T'=u'Tx_i,w_{i+1},x_{i+2}Tv'$.  Then in $Q(T')$ there is a $p$-block with $p\ge6$, a contradiction as shown on the left in Figure \ref{fig:lemma5case2}.  So assume $w\sim\{w_i,x_{i+2},x_{i+3}\}$.
Consider a path $T''=u'Tx_i,w_{i+1},w,x_{i+3}Tv'$.  Observe that the edge $x_iw_{i+1}$ is a nontrivial cut-edge in $Q(T'')$ unless there is a vertex  $w'$ adjacent to $x_i, w_{i+1}$ and another vertex of $T''$.
This third vertex is either $x_{i-1}$ or $w$. It could not be $x_{i-1}$ as shown before.
\begin{figure}[h]
\begin{center}
\includegraphics{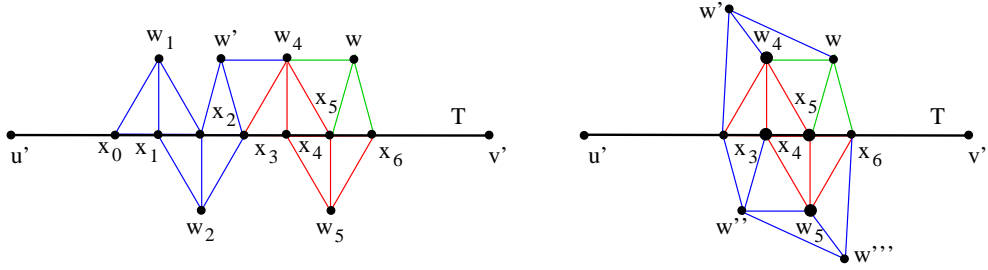}
\end{center}
\caption{Observe why there is no $w$ such that $w\sim\{w_{i+1}, x_{i+2}, x_{i+3}\}$.}
\label{fig:lemma5case2}
\end{figure}
Thus, $w'\sim \{x_i, w_{i+1}, w\}$.  See the right hand side of Figure \ref{fig:lemma5case2}.
Similarly, by considering the path $u'Tx_i, x_{i+1}, w_{i+2}, x_{i+3}Tv'$, we see there is a vertex $w'' \sim \{x_i, x_{i+1}, w_{i+2}\}$.
Finally, by considering the path $u'Tx_i, w'', w_{i+2}, x_{i+3}Tv'$, we have a vertex $w'''\sim \{w'', w_{i+2}, x_{i+3}\}$.
But now the graph $G[X]$, %$$G[\{x_i, x_{i+1}, x_{i+2}, x_{i+3},  w_{i+1}, w_{i+2}, w, w', w'', w'''\}]$$
where $X=\{x_{i+1},x_{i+2},w_{i+1},w_{i+2}\}$, gives the configuration $D(X)$ in $G$, as seen in Figure \ref{fig:lemma5case2} where the bold vertices represent $X$.

We see now, that a graph $H(S,u,v)$ in all the cases above was constructed by taking the union of $uSu'$, $v'Sv$, and a graph $H'$ induced by a shortest $(u',v')$-path $T$ (of length $20$) and, perhaps some vertices at distance $1$ from $T$.
Thus, any vertex of $H'$ is at distance at most $21$ to $v'=v_c$.
\end{proof}

\section{Proofs of Theorems} \label{Proofs}
\begin{proof}[Proof of Theorem \ref{two}]
Note if the two precolored vertices are adjacent, then the coloring is extendable by Thomassen's theorem.
In general, we use induction on $|V(G)|$ where the base case is precolored $u$ and $v$ connected by an edge.
Assume $G$ is connected, otherwise the result follows trivially by induction.

\noindent\textit{Claim.} $G$ has no separating $C_3$ or $C_4$.\\
Let $U$ be a vertex set of such a separating cycle. By the assumption of the theorem, $U$ does not separate $\{u,v\}$.
Let $V_1$ and $V_2$ be the vertex sets of disconnected plane graphs obtained by removing $G[U]$ from $G$, such that $\{u, v\} \subseteq V_1\cup U$.
By induction, color $G[V_1\cup U]$ from $L$.  This gives a proper coloring $c$ of $U$.
Now, in $G[V_2\cup U]$, there is a face with vertex set $U$ having color lists of size $1$ and all other vertices have color lists of size $5$.
Thus, by Theorem \ref{bms}, $G[V_2\cup U]$ is colorable from the corresponding lists.

Let $S={v_0,v_1,\ldots,v_m}$ be a shortest $(u,v)$-path in $G$, with $v_0=u$ and $v_m=v$, for $m\ge2$.
By Lemma \ref{nice1} there is a nice coloring $c$ of $v_0, \ldots, v_{m-2}$.
By Lemma \ref{properties-shortest-path}(4) there is at most one vertex adjacent to $v_{m-2}, v_{m-1}, v_m$ and at most one vertex adjacent to $v_{m-3}, v_{m-2}, v_{m-1}$, if $m\ge3$.
Let $c(v_{m-1}) \in L(v_{m-1})-(\{c(v_{m-2})\}\cup L(v_m))$.

If there is no vertex $x$, with $x\sim \{v_{m-2}, v_{m-1}, v_m\}$, and no vertex $x$, with $x\sim \{v_{m-3}, v_{m-2}, v_{m-1}\}$, then $c$ is a nice coloring of $S$.

Assume that there is a vertex $y$, with $y\sim \{v_{m-3}, v_{m-2}, v_{m-1}\}$, and there is no vertex $x$, with $x\sim \{v_{m-2}, v_{m-1}, v_m\}$, or, the other way around, there is no vertex $x$, with $x\sim \{v_{m-3}, v_{m-2}, v_{m-1}\}$ and there is a vertex $y$, with $y\sim \{v_{m-2}, v_{m-1}, v_m\}$.
Then $c$ is a proper coloring of $S$ such that $|L_c(p)|\geq 3$ for every $p\in N(S) -\{y\}$, and $|L_c(y)| \geq 2$.
Deleting $S$ and the corresponding colors from the lists of their neighbors in $G-S$ produces a list assignment where all vertices in a face containing $N(S)$ have lists of size at least $3$ (except for $y$), and all other vertices have lists of size $5$. Using Thomassen's theorem, $G-S$ can be colored from these lists.  Together with the coloring $c$ of $S$, it gives a proper $L$-coloring of $G$.

Finally, assume there is a vertex $x$, with $x\sim \{v_{m-3}, v_{m-2}, v_{m-1}\}$, and there is a vertex $w$, with $w\sim \{v_{m-2}, v_{m-1}, v_m\}$. Note that there is at most one additional vertex adjacent to $v_{m-1}$ and $v_m$, call it $z$ if it exists.
Delete $S$ from $G$ and add two new adjacent vertices $t$ and $s$ in the resulting face, also add edges $xt, ws, tz, sz,ty_i$, where $y_i\in N(v_{m-1})$ and $sx_i$, where $x_i\in N(v_m)$. Choose two new colors $\alpha$ and $\beta$ not used in any of the lists assigned to vertices of $G$.  Let $L'(t)=\{\alpha\}$, $L'(s) = \{\beta\}$, $L'(y_i)=L_c(y_i)\cup\{\alpha\}$, $L'(x_i)=L_c(x_i)\cup \{\beta\}$, $L'(z)=L_c(z)\cup \{ \alpha, \beta\}$, $L'(x)=L_c(x) \cup \{\alpha\}$, and $L'(w)= L_c(w) \cup \{\beta\}$. For every other vertex of this modified graph, let $L'$ be equal to $L_c$.  See Figure \ref{fig:thm2} for an illustration of this process.
\begin{figure}[h]
\begin{center}
\includegraphics{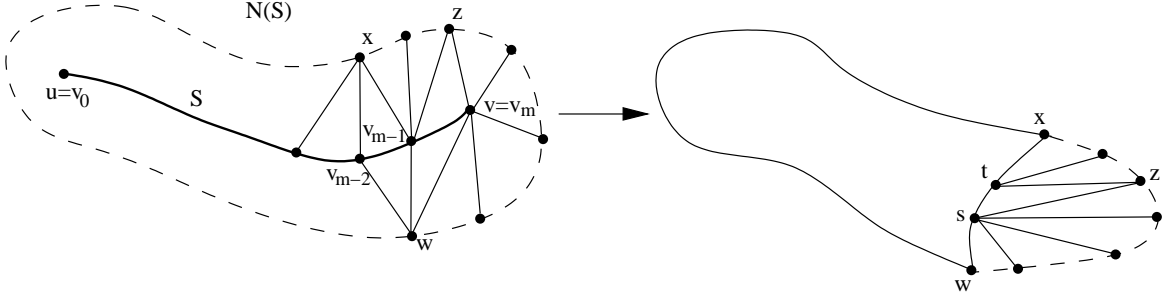}
\caption{The addition of vertices $t$ and $s$ in $G-S$.}
\label{fig:thm2}
\end{center}
\end{figure}
Observe that $L'$ satisfies the conditions of Thomassen's theorem, so there is a proper $L'$-coloring of this graph.  Thus, there is a proper $L'$-coloring of $G-S$, where no vertex uses colors $\alpha$ or $\beta$.
This is a proper $L_c$-coloring of $G-S$. Together with the coloring $c$ of $S$, it gives a proper $L$-coloring of $G$.
\end{proof}

\begin{proof}[Proof of Theorem \ref{Steiner}]
Let $T$ be a $(P,45)$-Steiner tree in $G'$, a reduced graph of $G$ satisfying the conditions of the theorem.
Let $L$ be an assignment of lists of colors to vertices of $G$ such that $|L(v)| = 1$ for $v\in P$ and $|L(v)|=5$ for $v\not\in P$.
We first color $G'$, then extend it to a proper $L$-coloring of $G$.

To color $G'$, first color special vertices of $T$ which are not in $P$ arbitrarily from their lists.
Let $\cal S$ be the set of branches in $T$ and let $S\in \cal S$ with endpoints $u_S,v_S$.
Let $H(S,u_S, v_S)=H(S)$ be the graph obtained by applying Lemma \ref{nice4} to $S$ and $c_S$ be a nice coloring of $H(S)$ from the corresponding lists (see Figure \ref{fig:H}).
Finally, let $c$ be a coloring of $H=\cup_{S\in {\cal S} } H(S)$, such that $c(v) =c_S(v) $ if $v\in H(S)$.

\noindent{\it Claim 1.} The coloring $c$ is a nice coloring of $H$.

Let $x, x'$ be two vertices of $H$ that do not belong to the same $H(S)$.
We shall prove that $x$ and $x'$ do not have common neighbors outside of $H$ and they are not adjacent.
Let $x \in H(S)$, $x'\in H(S')$, $S, S' \in {\cal S}$, $S\neq S'$.

If $x, x' \in V(T)$, then $x$ and $x'$ do not have a common neighbor outside of $T$ and they are not adjacent by part $(4)$ of the definition of a $(P,d)$-Steiner tree.

If $x\in V(T), x' \not\in V(T)$, then $x' \in V(H(S'))-V(S')$, thus $dist(x', v_{c'}) \leq 21$, where $v_{c'}$ is a center of $S'$, as
follows from Lemma \ref{nice4}.  From part $(3)$ of the definition of a $(P,d)$-Steiner tree,  we have that $dist(v_{c'}, x) \geq d$. Thus $dist(x, x') \geq d -21\geq 3$ when $d\geq 24$.

Finally if $x, x' \not\in V(T)$, then $x\in V(H(S))-V(S)$ and $x' \in V(H(S'))-V(S')$. Thus $dist(x, v_c), dist(x',v_{c'}) \leq 42$, where
$v_c, v_{c'}$ are centers of $S$ and $S'$, respectively.  Moreover $dist(v_c, v_{c'}) \geq d$. Thus $d(x, x') \geq d-42 \geq 3 $ if $d\geq 45$.

It follows that $c$ is a proper coloring of $H$.
To show that $c$ is nice, consider a vertex $v$ adjacent to $H$.  We see that  $v$ is adjacent to non-special vertices of $H(S)$ for at most one branch $S$ of $T$.
Since $c$ is a nice coloring of $H(S)$, it follows that $|L_c(v)|\geq 3$.

To conclude the proof of Claim 1, recall that $H$ is a connected graph containing all vertices of $P$.
Proposition \ref{nicelyH} implies that $G'$ is colorable from $L$.
To show that $G$ is colorable, it is sufficient to observe the following.

\begin{figure}[h]
\begin{center}
\includegraphics{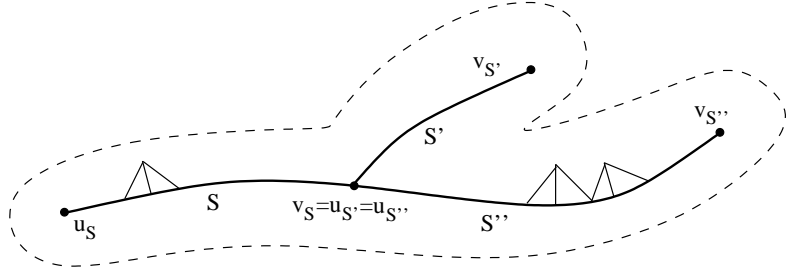}
\caption{An example of the graph $H$ obtained in the proof of Theorem \ref{Steiner}.}
\label{fig:H}
\end{center}
\end{figure}

\noindent{\it Claim 2.} Let $F$ be a graph, $P$ be a set of vertices, and $L$ be an assignment of lists of size $5$ to vertices of $V(G)-P$ and lists of size $1$ to vertices of $P$.  Let  $F' = R(F)$ be a reduction of $F$.
If $F'$ has a proper coloring from lists $L$ then $F$ has a proper coloring from lists $L$.

Let $c$ be a proper coloring of $F'$ from lists $L$.

If $F'$ was obtained from $F$  by removing the vertices in a region separated by $C_3$ or $C_4$, these vertices can be colored properly from $L$ using Theorem \ref{bms}.

If $F'$ was obtained from $F$ by removing the set $X$ of $4$ vertices, $y_1, y_2, z_1, z_2$ of configuration $D$, we see that $|L_c(y_i)| \geq 2$, $i = 1,2$ for the two vertices $y_1, y_2$ of degree two in $F[X]$ and $|L_c(z_i)|\geq 3$, $i = 1,2,$ for the two vertices $z_1,z_2$ of degree three in $F[X]$.  In the subgraph $F[X]$ each vertex has list size equal to its degree under list assignment $L_c$.  An $L_c$-coloring  of $F[X]$ can be found directly or by the results of \cite{KSW, T3}. Thus $F$ has a proper coloring from lists  $L$.

If $F'$ was obtained from $F$ by removing the set $X$ of  $7$ vertices $w, x_1, \ldots, x_6$ of configuration $W$, then we see that $|L_c(x_1)|, |L_c(x_4)|\geq 2$, $|L_c(x_2)|, |L_c(x_3)|, |L_c(x_5)|, |L_c(x_6)| \geq 3$, and $|L_c(w)|=5$.
Let $\alpha\in L_c(w)-(L_c(x_1)\cup L_c(x_4))$, so color $w$ with  $\alpha$ and remove $\alpha$ from $L_c(x_2),L_c(x_3),L_c(x_5),L_c(x_6)$.  What remains to be colored is a $6$-cycle with vertices having lists of size at least $2$, which is colorable by the classification of all $2$-choosable graphs by Erd\H{o}s et al. \cite{ERT}. Since $F[X]$ is properly colorable from lists $L_c$, $F$ is properly colorable from lists $L$.

This proves Claim 2.

Since $G'$ was obtained from $G$ via a sequence of reductions, the theorem follows.
\end{proof}

\begin{proof}[Proof of Theorem \ref{one-face}]

(1) Let $L$ be an assignment of lists of colors to vertices of $G$ such that $|L(x)|=5$ for all $x\not\in P$ and $|L(v_i)|=1$ for all $v_i\in P$. If $P$ is a set of vertices and edges with pairwise distance at least $3$, then for all $x \not\in P$, $x$ is adjacent to at most two vertices of $P$.
Thus, for every proper coloring $c$ of $G[P]$ from the corresponding lists $L$ and for all $x\not\in P$, we have $|L_c(x,P)|\geq 3$.  Moreover, $N(P)$ belongs to the frontier of a face in $G-P$.
Thus, by Proposition \ref{nicelyH}, $G$ is colorable from lists $L$.

(2) Without loss of generality, assume $C$ is on the unbounded face of $G$.  Let $P=\{v_0,v_1,\ldots,v_{k-1}\}\subseteq C$ be a set of at most six precolored vertices on the boundary of $C$.  Fix an assignment $L$ of lists of colors to the vertices of $G$ with $|L(v)|=5$ for all $v\in V(G)-P$ and $|L(v_i)|=1$ for all $v_i\in P$.  We shall show that $G$ is $L$-colorable provided the three forbidden configurations are not present.

We shall create a new graph $G'$ on the vertex set of $G$ with new lists $L'$.
Let $c_0, \ldots, c_{k-1}$ be distinct colors not present in $L(v)$ for any $v\in V(G)$.
Let $L'$ be a new list assignment with $L'(v_i):= \{c_i\}$ for $i= 0, \ldots, k-1$ and $L'(v) = L(v)-S_v \cup S'_v$ for each $v\in V(G)-P$, where $S_v$ is the set of colors used in lists $L$ of vertices in $P \cap N(v)$ and $S'_v $ is an arbitrary subset of the set of colors used in lists $L'$ of vertices of $P\cap N(v)$, such that $|S'_v|=|S_v|$.
In creating $L'$ we simply replaced the colors originally assigned to $P$ with new distinct colors, and replaced the old colors in the lists of vertices in the neighborhood of vertices of $P$.

Let a new plane graph $G'$ be obtained from $G$ by removing the edges $v_iv_{i+1}$ for $i=0, \ldots, k-1$
that correspond to non-consecutive vertices of $C$, and adding all edges $v_iv_{i+1}$ for $i=0, \ldots, k-1$ in the unbounded face of $G$.
The resulting graph has a new unbounded face with vertex set $P$, and, perhaps, some new edges.
By Theorem 7, $G'$ is $L'$-colorable by a coloring $c$ provided the three forbidden configurations are not present. Moreover, for any $v\not\in P$, we have $c(v) \not\in \{c_0, \ldots, c_k\}\cup S_v$, so $c(v) \in L(v)$ and $c(v) \not\in L(v_i)$ if $v\sim v_i$.  To create a proper $L$-coloring of $G$, replace the color $c_i$ with an element of $L(v_i)$ for $i=0, \ldots, k-1$.
\end{proof}

\begin{proof} [Proof of Theorem \ref{two-faces}]
Delete $P$ and the corresponding colors from the lists of adjacent vertices.
There are at most two faces, $F_1'$ or $F_1'$ and $F_2'$, in the graph $G-P$ such that the vertices adjacent to $P$ in $G$ belong to the boundaries of these two faces.  These vertices have lists of size at least $4$, and all other vertices in $G-P$ have lists of size at least $5$.
Call the resulting lists $L'$.
Add a vertex $v_i$ to the face $F_i'$ and make it adjacent to all vertices on $F_i'$,  $i = 1$, or $i=1,2$.
Let $\alpha$ be a color not used in any of the lists $L(v)$, $v\in V$.  Let $L''(v_1) = L''(v_2) = \{\alpha\}$,  $L''(v)= L'(v)\cup \{\alpha\}$, if $v\in V(F_1'\cup F_2')$ and $|L'(v)|=4$.
For all other vertices, let $L''(v)=L'(v)$.
Applying Theorem \ref{two} to the resulting graph with lists $L''$ allows for this graph to be properly colored from these lists.  We note here that it is not hard to see that this new graph does not contain any $\{v_1,v_2\}$-separating $C_3$s or $C_4$s because such a separating $C_3$ or $C_4$ would have to be made up of vertices and edges from the original graph and would have separated some of the precolored vertices of $G$, a contradiction.
This coloring gives a proper coloring of $G-P$ from lists $L'$, and thus it gives a proper coloring of $G$ from lists $L$.
\end{proof}

%%%%%%%%%%%%%%%%%%%%%%%%%%%%%%%%%%%%%%%%%%%%%%%%%%%%%%%%%%%%%%%%%%%%%%%%%%%%%%%%%

\section{Conclusions} \label{Conclusions}
We proved the question of Albertson has a positive answer if there are no short cycles separating precolored vertices and there is a nice tree containing precolored vertices.

We note here that by the definition of a $(P,d)$-Steiner tree, Theorem \ref{Steiner} can be applied to plane graphs with precolored vertices that are not far apart.  For example, let $G$ be a 100-cycle with vertices $v_0,v_1,\ldots,v_{99}$ and $P=\{v_1,v_{50},v_{98}\}$.  Then $G$ contains a $(P,48)$-Steiner tree obtained from deleting $v_0,v_{99}$ and incident edges.  The centers of the branches are far apart, but $dist(v_1,v_{98})=3$.

We believe that in a planar triangulation either such a tree could always be found, or there are small reducible configurations such as shown in Figure \ref{fig:reducible}.  The reducible configurations $D$ and $W$ are just two in a family of many reducible configurations of those types.  Modifying the definition of a reduced graph to include the removal of every reducible $K_4-e$ and every reducible $6$-wheel leads us to the following question.

\begin{question}
Is it the case that every reduced planar triangulation with a set $P$ of precolored vertices with $dist(P)\ge1000$ contains a $(P,45)$-Steiner tree?
\end{question}

If the above question has a positive answer, then by Theorem \ref{Steiner}, the precoloring of $P$ extends to a $5$-list coloring of $G$.  We did not strive to improve the constants here.  With more careful calculations, one could easily obtain smaller constants.

The condition of no separating short cycles seems to be essential.
Reducing the sizes of the lists, increasing the sizes of lists on so-called ``precolored'' vertices, or eliminating the distance condition in this problem is not possible even for a small number of precolored vertices, see Figure \ref{fig:distsize}.
Figure \ref{fig:tentacles} shows we cannot reduce the sizes of the lists, even if the vertices are on the same face.  This graph belongs to a family of graphs where the length of each path along the unbounded face from the outer triangles to the inner triangle must be divisible by three.  It is not hard to see that if the vertices of the inner triangle are assigned colors $1,2,3$, respectively, then one of the vertices with lists of size $2$ cannot be colored.
\begin{figure}[h]
  \centering
  \subfloat[][Non-extendable precoloring of three vertices at distance $2$.]{\label{fig:3precolor}\includegraphics{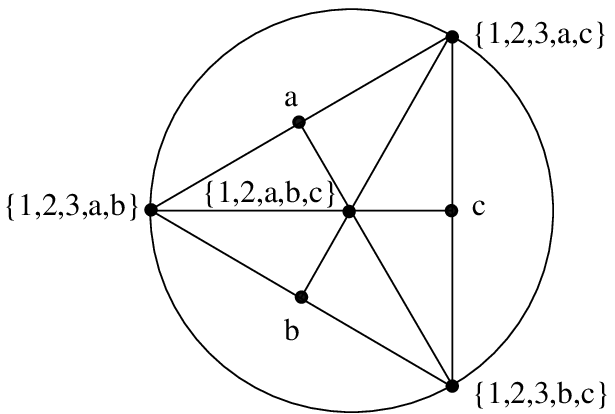}}
  \quad
  \subfloat[][Non-extendable precoloring of two vertices at distance $2$ where other vertices have lists of size $3$.]{\label{fig:2pre3lists}\includegraphics{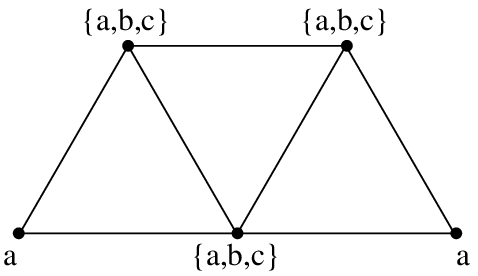}}
  \quad
  \subfloat[][Non-list-colorable graph with lists of size $1$, $2$, and $3$.]{\label{fig:123lists}\includegraphics{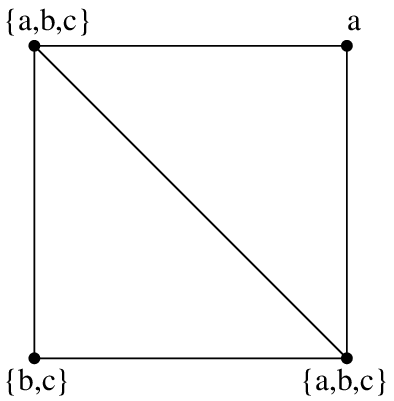}}
  \quad
  \subfloat[][Non-list-colorable graph with all other lists $\{1,2,3\}$.]{\label{fig:tentacles}\includegraphics{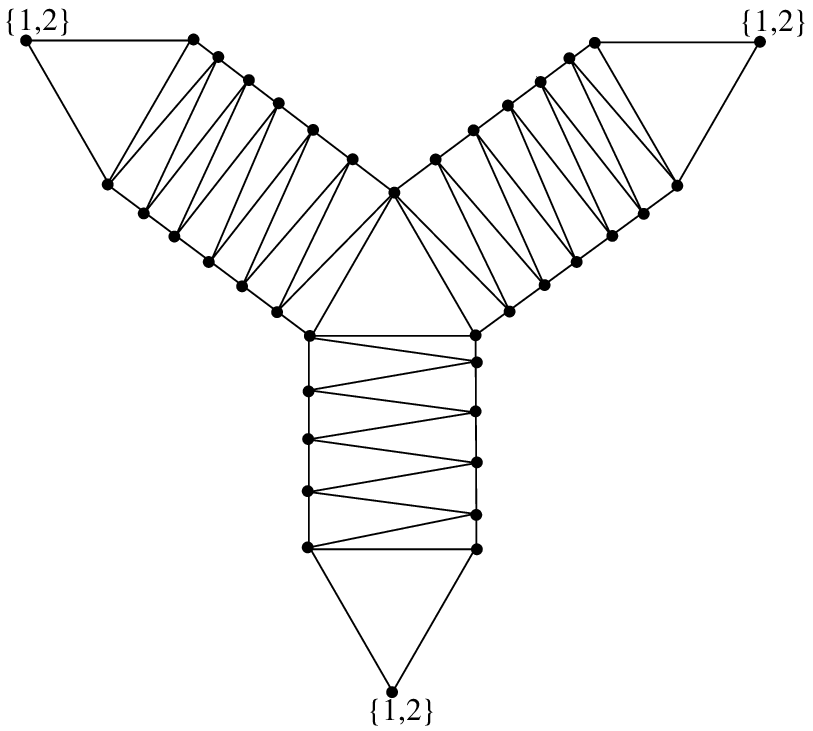}}
  \caption{Non-extendable precolorings.}
  \label{fig:distsize}
\end{figure}
However, we conjecture that a precoloring of two far-apart vertices is always extendable to a $5$-list coloring of a planar graph.

\section{Acknowledgements}
The authors wish to thank the anonymous referee for many helpful suggestions, especially an observation that greatly simplified and improved the proof of Theorem \ref{one-face} (2).

\end{document}